\documentclass[12pt]{amsart}

\usepackage{mathrsfs}
\usepackage{amssymb}
\usepackage{latexsym}

\usepackage[mathscr]{eucal}

\usepackage[all]{xy}
\usepackage{pb-diagram,pb-xy}
\usepackage{pstricks}

\usepackage[a4paper, hmargin=3.0cm, tmargin=3.2cm, bmargin=3.2cm]{geometry}
\numberwithin{equation}{section}

\newtheorem{mainthm}{Theorem}

\theoremstyle{plain}
\newtheorem{thm}[equation]{Theorem}
\newtheorem{lem}[equation]{Lemma}
\newtheorem{prop}[equation]{Proposition}
\newtheorem{cor}[equation]{Corollary}

\theoremstyle{definition}
\newtheorem{defn}[equation]{Definition}

\newcommand{\bC}{\mathbb{C}}
\newcommand{\bN}{\mathbb{N}}
\newcommand{\bP}{\mathbb{P}}
\newcommand{\bQ}{\mathbb{Q}}
\newcommand{\bR}{\mathbb{R}}
\newcommand{\bS}{\mathbb{S}}
\newcommand{\bZ}{\mathbb{Z}}
\newcommand{\cR}{\mathcal{R}}

\newcommand{\cB}{\mathcal{B}}
\newcommand{\cL}{\mathcal{L}}
\newcommand{\cP}{\mathcal{P}}

\newcommand{\fg}{\mathfrak{g}}
\newcommand{\ft}{\mathfrak{t}}
\newcommand{\fL}{\mathfrak{L}}

\newcommand{\cp}{\mathbb{C}\mathbb{P}}
\newcommand{\ch}{\operatorname{ch}}
\newcommand{\ph}{\operatorname{ph}}
\newcommand{\cotanh}{\operatorname{coth}}
\newcommand{\dash}{^{\prime}}
\newcommand{\Diff}{\operatorname{Diff}}
\newcommand{\dual}{^{\vee}}
\newcommand{\ev}{\operatorname{ev}}

\newcommand{\im}{\operatorname{Im}}
\newcommand{\id}{\operatorname{id}}
\newcommand{\inc}{\operatorname{inc}}
\newcommand{\loopinf}{\Omega^{\infty}}
\newcommand{\map}{\operatorname{map}}

\newcommand{\MTSO}{\operatorname{MTSO}}
\newcommand{\proj}{\operatorname{pr}}
\newcommand{\PT}{\operatorname{PT}}
\newcommand{\sign}{\operatorname{sign}}
\newcommand{\suspinf}{\Sigma^{\infty}}

\newcommand{\thom}{\operatorname{th}}

\newcommand{\Pont}{\operatorname{Pont}}
\newcommand{\res}{\operatorname{res}}
\newcommand{\bTh}{\mathbb{T}\mathbf{h}}
\newcommand{\trf}{\operatorname{trf}}
\newcommand{\trg}{\operatorname{trg}}

\newcommand{\Sym}{\operatorname{Sym}}
\newcommand{\rank}{\operatorname{rank}}
\newcommand{\spann}{\operatorname{span}}
\newcommand{\Gr}{\operatorname{Gr}}
\newcommand{\hol}{\operatorname{hol}}

\address{Mathematisches Institut der Universit\"at Bonn,
Endenicher Allee 60, 53115 Bonn, Bundesrepublik Deutschland}

\email{ebert@math.uni-bonn.de}


\begin{document}
\vspace*{-1cm}
\title[MMM classes]{Algebraic independence of generalized MMM-classes}

\author{Johannes Ebert}

\begin{abstract}
The generalized Morita-Miller-Mumford classes of a smooth oriented
manifold bundle are defined as the image of the characteristic
classes of the vertical tangent bundle under the Gysin homomorphism.
We show that if the dimension of the manifold is even, then all
MMM-classes in rational cohomology are nonzero for some bundle. In
odd dimensions, this is also true with one exception: the MMM-class
associated with the Hirzebruch $\cL$-class is always zero. We also show a
similar result for holomorphic fibre
bundles.
\end{abstract}

\maketitle
\setcounter{tocdepth}{1}
\tableofcontents

\section{Introduction and statement of results}

Let $M$ be a closed oriented $n$-dimensional smooth manifold and let
$\Diff^+ (M)$ be the topological group of all orientation-preserving
diffeomorphisms of $M$, endowed with the Whitney
$C^{\infty}$-topology. A \emph{smooth oriented $M$-bundle} is a
fibre bundle with structural group $\Diff^+ (M)$ and fibre $M$. Let
$ Q \to B$ be a $\Diff^+ (M)$-principal bundle. The \emph{vertical tangent bundle} of the smooth oriented $M$-bundle $f:E :=Q \times_{\Diff^+ (M)} M \to B$
is the oriented $n$-dimensional vector bundle $T^f=T_v E := Q
\times_{\Diff^+ (M) } TM \to E$.
A \emph{smooth oriented closed fibre bundle of dimension $n$} is a
map $f: E \to B$ such that for any component $C \subset B$, $f:
f^{-1} (C) \to C$ is a smooth oriented $M$-bundle for some closed
oriented $n$-manifold $M$. We sometimes abbreviate this term to
\emph{smooth fibre bundle}, because all manifold bundles we consider
are oriented and have closed fibres.

If $f: E \to B$ is a smooth oriented fibre bundle, then the
\emph{Gysin homomorphism} (or umkehr homomorphism) $f_{!}: H^{*}(E)
\to H^{*-n} (B)$ is defined (all cohomology groups in this paper
have rational coefficients, unless we explicitly state the
contrary). Define a linear map

\begin{equation}\label{kappaofe}
\kappa_{E}: H^* (BSO(n);\bQ) \to H^{*- n} (B ; \bQ)
\end{equation}

by

\begin{equation}\label{kappaofc}
\kappa_E (c) :=f_{!} (c(T_v E)) \in H^{k-n} (B); \; c \in H^{*}
(BSO(n);\bQ).
\end{equation}

The universal $M$-bundle $E_M \to B \Diff^+ (M)$ gives a map

\[
\kappa_{E_M}: H^* (BSO(n);\bQ) \to H^{*- n} (B \Diff^+ (M); \bQ).
\]

The homomorphism $\kappa_E$ is natural in the sense that $h^* \circ
\kappa_E = \kappa_{h^* E}$ for any map $h$ and so the images of $\kappa_E$ can be viewed as
characteristic classes of manifold bundles, which we call \emph{generalized Morita-Miller-Mumford
classes} or short MMM-classes. Morita \cite{Mor}, Miller \cite{Mil} and Mumford
\cite{Mum} first studied these classes in the $2$-dimensional case.

For a graded vector space $V$ and $n \in \bN$, we denote by
$\sigma^{-n} V$ the new graded vector space with $(\sigma^{-n}V)_m =
0 $ if $m \leq 0$ and $(\sigma^{-n} V)_m = V_{m+n}$ for $m
> 0$.
Then $\kappa_{E}$ becomes a map $\sigma^{-n} H^* (BSO(n); \bQ) \to
H^* (B ; \bQ)$ of graded vector spaces.

Let $\cR_n$ be a set of representatives for the oriented
diffeomorphism classes of oriented closed $n$-manifolds (connected
or non-connected) and let $\cR_{n}^{0} \subset \cR_n$ be the set of
connected $n$-manifolds. Put

\[\cB_n := \coprod_{M \in \cR_n} B \Diff^+ (M); \; \cB_{n}^{0} = \coprod_{M \in \cR_{n}^{0}} B
\Diff^+ (M)\subset \cB_n.\]

There are tautological smooth fibre bundles on these spaces and
therefore we get maps of graded vector spaces

\begin{equation}\label{kappauniversal1}
 \kappa^n: \sigma^{-n} H^* (BSO(n); \bQ) \to H^* (\cB_{n};
\bQ); \;  \kappa^{n,0}: \sigma^{-n} H^* (BSO(n); \bQ) \to H^*
(\cB_{n}^{0}; \bQ);
\end{equation}

$\kappa^{n,0}$ is the composition of $\kappa^n$ with the restriction
map $H^*(\cB_{n}) \to H^*(\cB_{n}^{0})$. Here is our first main
result.

\begin{mainthm}\label{maintheorem1}

\begin{enumerate}
\item For even $n$, $\kappa^{n,0} :\sigma^{-n} H^* (BSO(n); \bQ) \to H^* (\cB_{n}^{0};
\bQ)$ is injective.
\item For odd $n $, the kernel of $\kappa^{n,0}:\sigma^{-n} H^* (BSO(n); \bQ) \to H^* (\cB_{n}^{0};
\bQ)$ is the linear subspace that is generated by the components
$\cL_{4d} \in H^{4d} (BSO(n); \bQ)$ of the Hirzebruch $\cL$-class
(for $4d > n$).
\end{enumerate}
\end{mainthm}

Equivalently, Theorem \ref{maintheorem1} says (for even $n$) that
for each $0 \neq c \in \sigma^{-n} H^* (BSO(n);\bQ)$, there is a
connected $n$-manifold $M$ and a smooth oriented $M$-bundle $f: E
\to B$ such that $\kappa_{E} (c) \neq 0 \in H^{*} (B)$. Similarly
for odd $n$.

Generalized MMM-classes of degree $0$ are also interesting: these
are just the characteristic numbers of the fibre. The linear
independence of those is a well-known classical result by Thom
\cite{thom} and therefore we only care about positive degrees.

For an arbitrary graded $\bQ$-vector space $V$ (concentrated in
positive degrees), we let $\Lambda V$ be the free graded-commutative
unital $\bQ$-algebra generated by $V$. If $A$ is a
graded-commutative $\bQ$-algebra, then any graded vector space
homomorphism $\phi: V \to A$ extends uniquely to an homomorphism
$\Lambda \phi: \Lambda V \to A$ of graded algebras such that
$\Lambda \phi \circ s = \phi$ where $s: V \to \Lambda V$ is the
natural inclusion. Therefore, the map $\kappa^n$ from
\ref{kappauniversal1} induces a homomorphism

\begin{equation}\label{kappauniversal2}
\Lambda \kappa^n : \Lambda \sigma^{-n} H^* (BSO(n); \bQ) \to H^*
(\cB_n; \bQ).
\end{equation}

Our second main result is about $\Lambda \kappa^n$.

\begin{mainthm}\label{maintheorem2}
\begin{enumerate}
\item If $n$ is even, then the map $\Lambda \kappa^n $ from \ref{kappauniversal2} is injective.
\item If $n $ is odd, then the kernel of $\Lambda \kappa^n$ is the ideal generated
by the components $\cL_{4d} \in H^{4d} (BSO(n); \bQ)$ of the
Hirzebruch $\cL$-class (for $4d > n$).
\end{enumerate}
\end{mainthm}

We show a similar result in the complex case. A holomorphic fibre
bundle of dimension $m$ is a proper holomorphic submersion $f:E \to
B$ between complex manifolds of codimension $-m$. By Ehresmann's
fibration theorem, $f$ is a smooth oriented fibre bundle (but the
biholomorphic equivalence class of the fibres is not locally
constant). The vertical tangent bundle $T_v E := \ker Tf$ is a
complex vector bundle of rank $n$ and for any $c \in H^* (BU(m))$,
we can define

\[
\kappa_{E}^{\bC} (c):= f_{!} (c(T_v E)) \in H^{*-2m}(B).
\]

\begin{mainthm}\label{holomorphiccase}
\begin{enumerate}
\item For each $ 0 \neq c \in \sigma^{-2m}H^{*} (BU(m))$, there exists
a holomorphic fibre bundle $f:E \to B$ of dimension $m$ on a
projective variety $B$ such that $f_{!} (c(T_v E)) \neq 0$.
\item For any $0 \neq c \in \Lambda \sigma^{-2m}H^{*} (BU(m))$, there
exists a holomorphic fibre bundle with $m$-dimensional fibres on an
open complex manifold such that $\Lambda \kappa_{E}^{\bC} (c) \neq
0$.
\end{enumerate}
\end{mainthm}

Note that it is far from obvious to say what the universal
holomorphic bundle should be. Therefore we do not formulate Theorem
\ref{holomorphiccase} in the language of universal bundles.

The results of this paper can be interpreted in the language of the
Madsen-Tillmann-Weiss spectra $\MTSO(n)$ \cite{GMTW}, as we will
briefly explain. By definition, $\MTSO(n)$ is the Thom spectrum of
the inverse of the universal vector bundle $L_n \to BSO(n)$. If $f:E
\to B$ is an oriented manifold bundle of fibre dimension $n$, then
the Pontrjagin-Thom construction yields a spectrum map
$\alpha^{\flat}: \suspinf B_+ \to \MTSO(n)$. The spectrum cohomology
of $\MTSO(n)$ is, by the Thom isomorphism, isomorphic to $H^{*+n}
(BSO(n))$. Therefore, $\alpha^{\flat}$ induces a map of graded
vector spaces $\sigma^{-n} H^{*} (BSO(n)) \to H^* (B)$, which is the
same as the map $\kappa_E$.

The adjoint of $\alpha^{\flat}$ is a map $\alpha:B \to \loopinf_0
\MTSO (n)$ and it induces an algebra map $H^{*} (\loopinf_0
\MTSO(n)) \to H^* (B)$. Under the classical isomorphism $H^*
(\loopinf_0 \MTSO(n);\bQ) \cong \Lambda H^{*>0} (\MTSO(n);\bQ)$,
this map corresponds to $\Lambda$. Apart from the breakthrough works
\cite{MT}, \cite{MW}, \cite{GMTW}, the characteristic classes of
manifold bundles related to $\MTSO(n)$ have been studied by several
authors \cite{GRW}, \cite{Sad}. Their methods, however, do not
suffice to show Theorems \ref{maintheorem1} and \ref{maintheorem2}.

In general, the construction of manifold bundles and the computation
of generalized MMM-classes are rather difficult problems. The only
difficult constructions which we need in the present paper are in
the $2$-dimensional case, and for that we rely entirely on
\cite{Mil} and \cite{Mor}. There are some other computations of
MMM-classes which we want to mention though we do not need them.

The MMM-classes of bundles with compact connected Lie groups as
structural groups are relatively easy to compute due to the
''localization formula'' of \cite{AB}. A special case is the case of
homogeneous space bundles of the form $BH \to BG$ where $H \subset
G$ are compact Lie groups. In that case, the MMM-classes can be
expressed entirely in terms of Lie-theoretic data. In \cite{AKU}, a
similar localization principle is applied to cyclic structural
groups.

The MMM-classes associated with \emph{multiplicative sequences} are
rather well understood because of the close relationship with genera
(i.e., ring homomorphisms from the oriented bordism ring to $\bQ$),
see e.g. \cite{HBJ}. The theory of elliptic genera shows that many
of these MMM-classes are nontrivial. Unfortunately, this is not
enough to establish Theorem \ref{maintheorem1}.

Another source of smooth fibre bundles with nontrivial MMM-classes
is the following result. If $M$ is an oriented manifold with
signature $0$, then there exists an oriented smooth fibre bundle $E
\to \bS^1$ such that $E$ is oriented cobordant to $M$. This was
established by Burdick and Conner (combine Corollary 6.3 of
\cite{Con} with Theorem 1.2 of \cite{Burd}) away from the prime $2$.
Another proof was given by W. Neumann \cite{Neum} based on a result
of J\"anich \cite{Jaen}. Let $0 \neq x \in H^{4k}(BSO(4k))$ be a
class that is not a multiple of the Hirzebruch class. Then there is
a $4k$-manifold $M$ with signature $0$ and $\langle x (TM);[M]
\rangle \neq 0$ and a fibre bundle $f:M \to \bS^1$ by the above
results. Then $f_{!} (x(T_v M)) \neq 0 \in H^1 (\bS^1)$. Therefore,
in all dimensions of the form $4k-1$, the statement of Theorem
\ref{maintheorem3} is true for classes of degree $1$.

In section \ref{outlinesection}, we give a detailed overview of the
proof of the main results. In the appendix, we recapitulate the
definitions and the relevant properties of the Gysin-homomorphism
and the related transfer. The rest of the paper contains the details
of the proof outlined in section \ref{outlinesection}.

\subsection{Acknowledgement}

The author wants to acknowledge the hospitality and generosity of
the Mathematics Department of the University of Copenhagen, which is
where this project was begun.

\subsection{Notations and conventions}

All cohomology groups in this paper have rational coefficients. When
$G$ is a topological group which acts on the space $X$, we denote
the Borel construction by $E(G;X):= EG \times_G X$. We furthermore
abbreviate $\cB_n := \coprod_{M \in \cR_n} B \Diff^+ (M)$. Our
notation of standard characteristic classes differs from the
customary one. We give them the actual cohomological degree they
have as an index. For an example, $p_4 (V) $ will denote what is
commonly known as the first Pontrjagin class of the real vector
bundle. We hope that this does not lead to confusion. If $x$ is an
element of a graded vector space, we denote its degree by $|x|$,
implicitly assuming that $x$ is homogeneous. Moreover, all sub
vector spaces $ W \subset V$ of a graded vector space are assumed to
be graded, in other words $W= \oplus_n W \cap V_n$. The dual space
of a vector space $V$ is always denoted by $V \dual$.

\section{Outline of the proof}\label{outlinesection}

The proof of Theorems \ref{maintheorem1}, \ref{maintheorem2} and
\ref{holomorphiccase} is an eclectic combination of several
computations. In this section, we give an outline. The cohomology of
$BSO(n)$ is well known:

\[
H^* (BSO(2m+1); \bQ) \cong \bQ [p_4, \ldots , p_{4m}] ; \;  H^*
(BSO(2m); \bQ) \cong \bQ [p_4,  \ldots , p_{4m}, \chi] /(\chi^2 -
p_{4m}).
\]

The cases $n =0,1$ of Theorem \ref{maintheorem1} are empty. The fact
that the subspace generated by the components of the Hirzebruch
$\cL$-class is contained in the kernel of $\kappa^n$ for odd $n$
follows from the multiplicativity of the signature in fibre bundles
of odd-dimension: If $f:E \to B$ is a smooth oriented fibre bundle
with odd-dimensional fibres and $B$ is a closed oriented manifold,
then $\sign (E) =0$. This was first mentioned by Atiyah \cite{Atfib}
(without proof), proven later by Meyer \cite{Mey}, L\"uck-Ranicki
\cite{LR} and the author \cite{Eb}. For further reference, we state
this result explicitly.

\begin{thm}\label{vanishing}
For odd $n$, the kernel of $\kappa^n$ contains the subspace that is
generated by the components $\cL_{4d} \in H^{4d} (BSO(n); \bQ)$ of
the Hirzebruch $\cL$-class (for $4d > n$).
\end{thm}

Theorem \ref{maintheorem1} shows that this is the only constraint.
Because the components of $\cL$ form an additive basis of $H^*
(BSO(3); \bQ)$, Theorem \ref{vanishing} forces $\kappa^3$ to be the
zero map. Thus Theorem \ref{maintheorem1} is also empty in the
$3$-dimensional case. The case $n=2$ is a classical result, which is the main ingredient for the proof of Theorem \ref{maintheorem1}.

\begin{thm}\label{morita}
The map $\kappa^{2,0} $ is injective.
\end{thm}

This was first established by Miller \cite{Mil} and Morita
\cite{Mor}. Today, there are other proofs by Akita-Kawazumi-Uemura
\cite{AKU} and Madsen-Tillmann \cite{MT}. Of course, the
affirmative solution of the Mumford conjecture by Madsen and Weiss \cite{MW} also
implies Theorem \ref{morita}.

We denote by $\Pont^* (n) \subset H^* (BSO(n); \bQ)$ the subring
generated by the Pontrjagin classes. If $V \to X$ is a real vector
bundle, then $\Pont (V) \subset H^* (X)$ is the subring generated by
the Pontrjagin classes of $X$. The main bulk of work to prove
Theorem \ref{maintheorem1} is:

\begin{thm}\label{maintheorem3}

\begin{enumerate}
\item For even $n$, $ \kappa^{n,0} :\sigma^{-n}\Pont^* (n) \to H^* (\cB_{n}^{0})$ is injective.
\item For odd $n $, the kernel of $\kappa^{n,0}:\sigma^{-n}\Pont^* (n) \to H^* (\cB_{n}^{0})$ is the
linear subspace that is generated by $\cL_{4d}$ (for $4d > n$).
\end{enumerate}
\end{thm}

Theorem \ref{maintheorem3} implies Theorem \ref{maintheorem2}: for
odd $n$, $\Pont^*(n)=H^*(BSO(n))$ and for $n=2m$, the argument is so
short and easy that we give it here. The total space of the unit
sphere bundle of the universal vector bundle on $BSO(2m+1)$ is
homotopy equivalent to $BSO(2m)$ and the bundle projection
corresponds to the inclusion map $f: BSO(2m) \to BSO (2m+1)$. This
map induces an isomorphism $\Pont^* (2m+1) \to \Pont^* (2m)$. Any
element $x \in H^{*} (BSO(2m); \bQ)$ can be written uniquely as $x =
f^* x_1   \chi + f^* x_2$ with $x_i \in H^{*} (BSO(2m+1) ; \bQ)$.
Lemma \ref{eulerclass} below and Theorem \ref{maintheorem3}
immediately imply Theorem \ref{maintheorem2}.

\begin{lem}\label{eulerclass}
Let $f:BSO(2m) \to BSO(2m+1)$ be the universal $\bS^{2m}$-bundle and
let $x =  f^* x_1    \chi + f^* x_2$ be as above. Then $ p_{!} (x (T_v
BSO(2m))) = 2 x_1$.
\end{lem}

\begin{proof}
The vertical tangent bundle $T_v BSO(2m)$ is isomorphic to the
universal $2m$-dimensional vector bundle. Therefore: $f_{!} (x (T_v
(E)) ) = f_{!} (x) = f_{!}( f^* x_1  \chi + f^* x_2) = x_1
f_{!}(\chi) + f_{!} (1) x_2 = 2 x_1 $, since $f_{!} (\chi) =
\chi(\bS^{2m}) = 2 $ and $f_{!} (1)=0$.
\end{proof}

The proof of Theorem \ref{maintheorem3} has two parts. The first
part is an induction argument, using Theorem \ref{morita} as
induction beginning and the second part deals with the classes that
are missed by the inductive argument.
The idea of the induction is straghtforward. Let $n$ be given. Let $f_i:E_i
\to B_i$ be manifold bundles of fibre dimension $n_i$, $i=1,2$, $n_1
+ n_2 =n$. The idea is to consider the product bundle $f= f_1 \times
f_2: E_1 \times E_2 \to B_1 \times B_2$, which has fibre dimension
$n$. The MMM-classes of the product can be expressed by the
MMM-classes of the two factors. It turns out that we can detect most, but not all MMM-classes on products of lower-dimensional manifold bundles. Here is
the exception.

Recall that the \emph{Pontrjagin character} of a real vector bundle
$V \to X$ is $\ph (V) := \ch (V \otimes \bC)$. Since $V \otimes \bC
\cong \overline{V \otimes \bC}$ (it is self-conjugate), it follows
that $\ph_{4d + 2} (V ) =0$, so $\ph$ is concentrated in degrees
that are divisible by $4$. In fact, $\ph_{4d} \in \Pont^{4d} (n)$, $
n = \rank (V)$. Note that if $V$ is itself complex, then $\ph (V) =
\ch (V \otimes_{\bR} \bC ) = \ch (V \oplus \overline{V})$.

\begin{prop}\label{inductionstep}

\begin{enumerate}
\item Let $n=2m$ be even and assume that Theorem
\ref{maintheorem3} has been proven for all even dimensions $2l < n$.
Then the kernel of $\kappa^{n,0}: \sigma^{-n} \Pont^* (n) \to H^*
(\cB_{n}^{0};\bQ)$ is contained in the span of the components
$\ph_{4d}$, $4d \geq n$.
\item Let $n=2m+1 \geq 7$ be odd and assume that Theorem
\ref{maintheorem3} has been proven for all dimensions less than $n$.
Then the kernel of $\kappa^{n,0}: \sigma^{-n} \Pont^* (n) \to H^*
(\cB_{n}^{0};\bQ)$ is contained in the span of the components
$\ph_{4d}$ and $\cL_{4d}$, $4d \geq 2m+1$.
\end{enumerate}
\end{prop}

The proof is given in section \ref{inductionstepsection}. By Proposition \ref{inductionstep} and Theorem \ref{morita}, two steps remain to be done for the proof of Theorem \ref{maintheorem3} and hence Theorem \ref{maintheorem1}. We have to show that $\kappa^{n}(\ph_{4d}) \neq 0$ for all $4d \geq n \geq 4$. Furthermore, we have to show the case $n=5$ of Theorem \ref{maintheorem3} from scratch.

There are two ideas involved: we do explicit computations for
bundles of complex projective spaces and then we use what we call ''loop
space construction'' to increase the dimension.

Let the group $SU(m+1)$ act on $\cp^m$ in the usual way. Consider
the Borel-construction $q:E(SU(m+1);\cp^m) \to BSU(m+1)$. In section
\ref{cpnsection}, we will show the following result.

\begin{thm}\label{hirzebruch2}
For all $d \geq k$, the class $\kappa_{E(SU(2k+1);\cp^{2k})}
(\ph_{4d}) \in H^{4d-4k}(BSU(2k+1))$ is nonzero.
\end{thm}

To finish the proof of Theorem \ref{maintheorem1} in
the even-dimensional case it remains to prove that $\kappa^{4k+2}(\ph_{4d})\neq 0$ if $4d \geq 4k+2 \geq 6$.
To do this, we employ the loop space construction that we describe now.

Let $M$ be an oriented closed $n$-manifold and $f:E \to X$ a smooth
oriented $M$-bundle. Let $LX$ be the free loop space of $X$ and let
$\ev: \bS^1 \times LX \to X$ be the evaluation map
$\ev(t,\gamma):=\gamma(t)$.

Consider the diagram ($\proj$ is the obvious projection):

\begin{equation}\label{loopspaceconstruction}
\xymatrix{
\fL E:= \bS^1 \times LX \times_{X} E \ar[d]^{ f \dash} \ar[r]^-{h}   &  E \ar[d]^{f}\\
 \bS^1 \times LX \ar[r]^{\ev} \ar[d]^{\proj} &  X  \\
LX   &   \\
}
\end{equation}

The composition on the left-hand side is denoted $\fL p:= \proj
\circ f \dash: \fL E \to LX$; this is a smooth oriented $\bS^1
\times M$-bundle. We call it the \emph{loop space construction} on
the bundle $E$.

The generalized MMM-classes of $\fL E \to LX$ can be expressed in
terms of those of $E \to X$. The result is that the following
diagram is commutative:

\begin{equation}\label{loopspacetransg}
\xymatrix{
\Pont^{*} (n+1) \ar[r] \ar[d]^{\kappa_{\fL E}} & \Pont^{*} (n) \ar[d]^{\kappa_{E}} \\
H^{*-n-1} (LX)   & H^{*-n} (X). \ar[l]_-{\trg} }
\end{equation}

The bottom map is the \emph{transgression}, see Definition
\ref{deftransgress}. We can of course iterate the loop space
construction. When we apply it $r$ times to the $M$-bundle $E \to
X$, we obtain an $(\bS^1)^r \times M$-bundle $\fL^r p :\fL^r E \to
L^r (X)= \map ((\bS^1)^r; X)$. Also, the transgression can be
iterated and gives $\trg^r : H^{*} (X) \to H^{*-r} (L^r X)$. Now let
$f:E \to X$ be an $M^{4k}$-bundle and let $4d \geq 4k + r$. Assume
that $f_{!} (\ph_{4d} (T_v E)) \in H^{4d-4k} (X)$ is nonzero. Since
$\ph_{4d}$ does not lie in the kernel of the restriction $H^*
(BSO(4d+r)) \to H^* (BSO(4n))$, the class $\kappa^{\fL^r E}
(\ph_{4d}) \in H^{4d-4k-r} (L^r X)$ is nontrivial provided that
$f_{!} (\ph_{4d} (T_v E)) \in H^{4d-4k} (X)$ does not lie in the kernel
of $\trg^r$.

For a general space $X$, the transgression is far from being
injective, but it is injective if $X$ is simply-connected and the
rational cohomology of $X$ is a free graded-commutative algebra,
compare \ref{transgressinjective}. If $X$ is an addition
$r$-connected, then $\trg^r$ is injective.

The base space $BSU(2k+1)$ of the universal $\cp^{2k}$-bundle in
Theorem \ref{hirzebruch2} is $3$-connected and its rational
cohomology is a polynomial algebra and so $\trg^r$ is injective for
$r=1,2,3$. Therefore, Theorem \ref{hirzebruch2} implies that
$\kappa^n (\ph_{4d}) \neq 0$ if $n =4k+r$ for $0 \leq r \leq 3$.
This concludes, by Proposition \ref{inductionstep}, the proof of
Theorem \ref{maintheorem1} in the even-dimensional case.

For the odd-dimensional case, the only thing that is left is the
induction beginning (in dimension $5$). This is accomplished by the
same method.

\begin{thm}\label{hirzebruch1}
Let $q: E(SU(3); \cp^2) \to BSU(3)$ be the Borel construction and
$d>0$. Then the kernel of $\kappa_{E(SU(3); \cp^2)}: \Pont^{4d+4}(4)
\to H^{4d}(BSU(3))$ is one-dimensional and spanned by $\cL_{4d+4}$.
\end{thm}

\begin{cor}
Let $\fL q :\fL E :=  \fL E(SU(3); \cp^2) \to L BSU(3)$ (it is an
$\bS^1 \times \cp^2$-bundle). Then the kernel of $(\fL q)_{!}
:\Pont^{*+5} (T_v \fL E) \to H^* (L BSU(3))$ is spanned by the
components of the Hirzebruch class.
\end{cor}

The corollary follows immediately from Theorem \ref{hirzebruch1},
diagram \ref{loopspacetransg} and Proposition
\ref{transgressinjective}. This gives the induction beginning and
finishes the proof of Theorem \ref{maintheorem1}. Actually, it is
quite surprising that a single $5$-manifold, namely $\bS^1 \times
\cp^2$, sufficed to detect all MMM-classes.

Once Theorem \ref{maintheorem1} is shown, Theorem \ref{maintheorem2}
is a rather formal consequence that uses the Barratt-Priddy-Quillen
Theorem on the infinite symmetric group. We will not give any
details here and refer to section \ref{lineartoalgebraicsection}
instead.

The proof of Theorem \ref{holomorphiccase} is a simple variation of
the proofs of Theorems \ref{maintheorem1} and \ref{maintheorem2} and
will be discussed in section \ref{holomorphic}.

\section{The induction step}\label{inductionstepsection}

In this section, we prove Proposition \ref{inductionstep}. First we
recall that the Hirzebruch $\cL$-class is the multiplicative
sequence in the Pontrjagin classes that is associated with the
formal power series

\begin{equation}\label{deflclass}
\sqrt{x} \cotanh (\sqrt{x}) = \sum_{d =0}^{\infty}\frac{2^{2k}
B_{2k}}{(2k)!} x^k,
\end{equation}

where $B_{2k}$ denote the Bernoulli numbers. It is crucial for our
proofs that $B_{2k} \neq 0$.

The main part of the proof is pure linear algebra. The Whitney sum
map $BSO(n_1) \times BSO(n_2) \to BSO(n)$ ($n_1 + n_2 =n$) induces a
map

\[r_{n_1,n_2} : \sigma^{-n} \Pont^* (n) \to \sigma^{-n_1}
\Pont^* (n_1) \otimes \sigma^{-n_2} \Pont^* (n_2).\]

Furthermore, we let $L(n) \subset \sigma^{-n} \Pont^* (n)$ be the
subspace spanned by the components of the Hirzebruch $\cL$-class.
For $n_1 + n_2 =n$, let $\tilde{r}_{n_1,n_2}$ be the composition

\[
\begin{split}
 \tilde{r}_{n_1,n_2} : \sigma^{-n} \Pont^* (n)
\stackrel{r_{n_1,n_2}}{\to} \sigma^{-n_1} \Pont^* (n_1) \otimes
\sigma^{-n_2} \Pont^* (n_2) \to \\
 \sigma^{-n_1} \Pont^*
(n_1) \otimes (\sigma^{-n_2} \Pont^* (n_2)) / L(n_2)
\end{split}
\]

with the quotient map.

\begin{lem}\label{symmetricfunctionlemma}
\begin{enumerate}
\item Let $n=2m$. Then the intersection $\bigcap_{m_1 + m_2 =m, 0 < m_1 < m} \ker
(r_{2m_1, 2m_2}) \subset \sigma^{-n} \Pont^* (n)$ is the vector
space spanned by the elements $\ph_{4d}$, $4d \geq n$.
\item Let $n = 2m+1 \geq 7$. Then the intersection $\bigcap_{m_1 + m_2 =m, 0 < m_1 < m} \ker
(\tilde{r}_{2m_1, 2m_2+1}) \subset \sigma^{-n} \Pont^* (2m+1)$ is
the vector space spanned by $\ph_{4d}$ and $\cL_{4d}$, $4d \geq n$.
\end{enumerate}
\end{lem}

\begin{proof}[Proof of Lemma \ref{symmetricfunctionlemma}, part 1]
We identify $\Pont^* (2m) = \bQ [x_1, \ldots , x_m]^{\Sigma_m}$,
where $x_1, \ldots , x_m$ are indeterminates of degree $4$, the
Pontrjagin classes correspond to elementary symmetric functions and
$\ph_{4d}$ to $ x_{1}^{d} + \ldots + x_{m}^{d}$. Let us introduce
some abbreviations. If $S = \{ i_1, \ldots ,i_s\} \subset
\underline{m}$, then $V_S := \bQ[x_{i_1}, \ldots, x_{i_s}]$.
Moreover, $V_{S}^{<d}$ denotes the subspace of element of degree
less than $d$ (and, as usual, all degrees are total degrees).

The kernel of $r_{2k,2n-2k}$ agrees, up to a degree shift, with the
kernel of the quotient map

\[
V_{\{1,\ldots ,m\}} \to \frac{V_{\{1, \ldots k\}} \otimes V_{\{k+1,
\ldots , m\}}}{V_{\{1, \ldots k\}}^{<2k} \otimes V_{\{k+1, \ldots ,
m\}} \oplus V_{\{1, \ldots k\}} \otimes V_{\{k+1, \ldots ,
m\}}^{<2m-2k}}
\]

which is the same as

\begin{equation}\label{bigintersection}
\bigcap_{k=1}^{m-1} V_{\{1,\ldots,k\}}^{<2k} \otimes
V_{\{k+1,\ldots,m\}} \oplus V_{\{1,\ldots,k\}} \otimes
V_{\{k+1,\ldots,m\}}^{<2m-2k}.
\end{equation}

Let $4d \geq 2m$. We have to show the following: if a homogeneous
symmetric polynomial $p(x_1 \ldots , x_m)$ of degree $4d$ lies in
the intersection \ref{bigintersection}, then $p$ must be a power
sum, i.e. a multiple of $\ph_{4d}$.

Let $\cP$ is the set of all partitions of the set $\underline{m}$
into two parts, $\underline{m}= S_1 \coprod S_2$. A \emph{symmetric}
polynomial $p$ lies in the intersection \ref{bigintersection} if and
only if it lies in

\begin{equation}\label{defineU}
U =\bigcap_{P \in \cP}  V_{S_1}^{<2|S_1|} \otimes V_{S_2} \oplus
V_{S_1} \otimes V_{S_2}^{<2|S_2|}.
\end{equation}

Clearly, each of the spaces in \ref{defineU} whose intersection is
$U$ is spanned by monomials. Therefore $U$ is spanned by monomials,
too. Therefore, the space $U$ has the following property: if $p \in
U $ is written as a linear combination of monomials $p=\sum_{i} a_i
p_i$ with pairwise distinct monomials $p_i$ and $0 \neq a_i \in
\bQ$, then $p_i \in U$.

We call the monomials of the form $x_{i}^{j}$ \emph{pure} and all
the other ones \emph{impure}. We will show that any monomial in $U$
is pure. This finishes the proof because then any symmetric $p \in
U$ must be a linear combination of pure monomials and the symmetry
forces $p$ to be a power sum.

Let now $p$ be an impure monomial of degree $4d \geq 2m$. We want to
show that $p$ does not lie in $U$. Without loss of generality
(symmetry!), we can assume that $p=x_{1}^{d_1} \ldots x_{k}^{d_k}$
and $0 < d_k \leq d_j$ for all $j = 1,\ldots,k$. Note that $k \geq
2$ since $p$ is impure. There are three cases to distinguish.

\begin{itemize}
\item If $k =m$, then $4(d-d_k)=4(d_1 + \ldots + d_{k-1}) \geq 4 (k-1)d_k \geq 4 (k-1) \geq
2(m-1)$. Then $f$ is not contained in the space
\[
V_{\{1,\ldots,m-1\}}^{<2(m-1)}  \otimes  V_{\{m\}} \oplus V_{\{1,
\ldots,m-1\}} \otimes V_{\{m\}}^{<2}
\]
and hence not in $U$.
\item If $k < m$ and $2d_k \geq m-k$, then $4(d-d_k) \geq 4 (k-1)
d_k \geq 2 (k-1)(m-k) \geq 2(k-1)$. Then $f$ is not contained in
\[
V_{\{1, \ldots k-1\}}^{<2(k-1)}  \otimes  V_{\{k, \ldots ,m\}}
\oplus V_{\{1, \ldots,k-1\}} \otimes V_{\{k, \ldots m\}}^{<2 (m-k)}
\]
and hence not in $U$.
\item If $k < m$ and $2d_k < m-k$. Put $e = 2d_k$. Then $2d_k \geq
e$ and $2(d-d_k) = 2d -e \geq m-e$ and thus $f$ does not lie in
\[
\begin{split}
V_{\{1, \ldots k-1, k+1, \ldots m-e+1 \}}^{<2(m-e)}
\otimes V_{\{k, m-e+2 ,\ldots ,m\}} \oplus \\
V_{\{1, \ldots k-1, k+1, \ldots m-e+1 \}} \otimes  V_{\{k, m-e+2
,\ldots ,m\}}^{<2e}
\end{split}
\]
and hence it is not in $U$ either.
\end{itemize}
\end{proof}

To show the second part of Lemma \ref{symmetricfunctionlemma}, we
need another lemma.

\begin{lem}\label{powerseries}
Let $f = \sum_{k=0}^{\infty} f_k x^k \in 1 + x \bQ [[x]]$ be a power
series such that $f_k \neq 0$ for all $k$. Let $F = \sum_{i \geq 0}
F_i$ be the corresponding multiplicative sequence. Let $m \geq 3$
and let $h( x_1, \ldots, x_m)$ be a symmetric homogeneous polynomial
of degree $d$. Assume that

\[
h(x_1, \ldots , x_m) = \sum_{i=0}^{d} a_i x_{m}^{i} F_{d-i}
(x_1,\ldots , x_{m-1}).
\]

Then $h(x_1, \ldots , x_m) = a_0 F_d (x_1,\ldots ,x_m)$.
\end{lem}

\begin{proof}
The assumption that $f_k \neq 0$ implies that $F_i$ is nonzero (for
all $i$ and an arbitrary positive number of variables). By definition of
multiplicative sequences, we can write

\[
h (x_1,\ldots ,x_m) = \sum_{k=0}^{d} \sum_{i+j=k} a_i f_j x_{m}^{i}
x_{m-1}^{j} F_{d-k} (x_1,\ldots ,x_{m-2})
\]

and by symmetry we also have

\[
h (x_1,\ldots ,x_m) = \sum_{k=0}^{d} \sum_{i+j=k} a_i f_j
x_{m-1}^{i} x_{m}^{j} F_{d-k} (x_1,\ldots ,x_{m-2}).
\]

Therefore $a_i f_j = a_j f_i$ for all $0 \leq i+j \leq d$
(here the assumption that $m \geq 2$ is essential). Thus $a_j = a_j
f_0 = f_j a_0$ and hence

\[
h(x_1, \ldots , x_m) = a_0 \sum_{i=0}^{d} f_i x_{m}^{i} F_{d-i}
(x_1,\ldots , x_{m-1}) =a_0 F_d (x_1,\ldots ,x_m).
\]

\end{proof}

\begin{proof}[Proof of Lemma \ref{symmetricfunctionlemma}, Part 2]
The space $\ker (\tilde{r}_{2m_1, 2m_2+1}) \subset \sigma^{-n}
\Pont^* (2m+1)$ is the sum of the space $\ker (r_{2m_1, 2m_2 + 1})$
(which was computed in the proof of the first part) and the space
$\bQ[x_1, \ldots, x_{m_1}] \otimes L(2m_2 +1)$. There is an
inclusion relation $\bQ[x_1] \otimes L(2m-1) \subset \bQ[x_1,\ldots,
x_{m_1}] \otimes L(2m_2 +1)$ for all $m_1 \geq 1$. Thus the
intersection agrees with the smallest space, i.e. $\bQ[x_1] \otimes
L(2m-1)$. A homogeneous symmetric polynomial $ f \in \bQ[x_1]
\otimes L(2m-1)$ of degree $4d$ can be written as $f(x_1,\ldots
,x_m) = \sum_{i=0}^{d} a_i x_{1}^{i} \cL_{4i} (x_2,\ldots ,x_m)$
with $a_i \in \bQ$. By Lemma \ref{powerseries}, $f$ is a multiple of
the Hirzebruch class (here we use that the coefficients of the power
series \ref{deflclass} are nonzero).
\end{proof}

\begin{proof}[Proof of Proposition \ref{inductionstep}]
Let $f_i:E_i \to B_i$, $i=1,2$, be two oriented fibre bundles of
fibre dimension $n_i > 0$ with $n_1 + n_2 =n$. Consider $f = f_1
\times f_2: E = E_1 \times E_2 \to B = B_1 \times B_2$, which is an
oriented fibre bundle of fibre dimension $n$. The umkehr
homomorphism is compatible with products in the sense that

\begin{equation}
(f_1 \times f_2)_{!} (x_1 \times x_2) = (f_1)_{!} (x_1) \times
(f_2)_{!} (x_2)
\end{equation}

for all $x_i \in H^* (E_i)$ (the signs are all $+1$ since $x_i$ has
even degree). Therefore the diagram

\begin{equation}
\xymatrix{ \sigma^{-n} \Pont^* (n) \ar[r]^{\kappa^{n,0}}
\ar[d]^{r_{n_1, n_2}} &
H^*(\cB_{n}^{0}) \ar[d]\\
\sigma^{-n_1} \Pont^* (n_1) \otimes \sigma^{-n_2} \Pont (n_2)
\ar[r]^-{\kappa^{n_1,0} \otimes  \kappa^{n_2,0}} &
H^*(\cB_{n_1}^{0}) \otimes  H^*(\cB_{n_2}^{0})\\
}
\end{equation}

is commutative; the left-hand side vertical map is induced by the
Whitney sum and the right hand side vertical map is induced by
taking product bundles. A straightforward application of Lemma
\ref{symmetricfunctionlemma} completes the proof.
\end{proof}

\section{The loop space construction}\label{loopspacesection}

\subsection*{The loop space construction}

Let $M$ be an oriented closed $n$-manifold and $f:E \to X$ a smooth
oriented $M$-bundle. Let $LX$ be the free loop space of $X$ and let
$\ev: \bS^1 \times LX \to X$ be the evaluation map
$\ev(t,\gamma):=\gamma(t)$. Moreover, $\eta: LX \to X$ is defined by
$\eta(\gamma)=\ev(1,\gamma)$. Recall that the loop space
construction $\fL p = \proj \circ f \dash$ is defined by the diagram
\ref{loopspaceconstruction}:

\begin{equation}
\xymatrix{
\fL E:= \bS^1 \times LX \times_{X} E \ar[d]^{ f \dash} \ar[r]^-{h}   &  E \ar[d]^{f}\\
 \bS^1 \times LX \ar[r]^{\ev} \ar[d]^{\proj} &  X  \\
LX .  &   \\
}
\end{equation}

\subsection*{Relation to loop groups}

There is an alternative view on the loop space construction which
might be illuminating though we do not need it in the sequel.

Let $G$ be a topological group and let $M$ be an oriented closed
$n$-manifold with a $G$-action. Let $f:E= E(G;M):=EG \times_G M \to
X=BG$ be the Borel construction, an oriented $M$-bundle.
The loop group $LG$ (multiplication is defined pointwise) acts on
$\bS^1 \times M$ by the formula

\[
\gamma \cdot (t,m):= (t, \gamma (t) m),
\]

where $\gamma \in LG$, $t \in \bS^1$, $m \in M$. Thus we get an
induced $\bS^1 \times M$-bundle $q: E(LG; \bS^1 \times M) \to BLG$.

Given an $LG$-principal bundle $Q \to X$, then $(Q \times \bS^1
\times G) / \sim \to \bS^1 \times X$, where $(q,t,g)\sim (q \gamma,
t, \gamma(t)g)$ for $\gamma \in LG$, is a $G$-principal bundle. In
the universal case $X = BLG$, the classifying map of this bundle is
a map $\phi: BLG \to LBG$, which is a homotopy equivalence if $G$ is
connected. It is not hard to see that there is a pullback-diagram

\[
\xymatrix{
E(LG ; \bS^1 \times M) \ar[r] \ar[d]^{q} & \fL E(G;M) \ar[d]^{\fL p}\\
 BLG \ar[r]^{\phi} &  L BG.
}
\]

\subsection*{Characteristic classes of the loop space construction}

Let us compute the generalized MMM-classes of the bundle $\fL f: \fL
E \to LX$ of \ref{loopspaceconstruction} in terms of those of the
original bundle $f: E \to X$. Let $x \in \Pont^* (n+1)$. We denote
the vertical tangent bundles by $T^{\fL f}$ and $T^f$ in
self-explaining notation. The vertical tangent bundle of $\fL f =
\proj \circ f \dash$ is seen to be isomorphic to $(f \dash)^*
T^{\proj} \oplus T^{f \dash} \cong \bR \oplus h^{*} T^{f}$.
Therefore

\begin{equation}\label{mmmclassesofloopspace1}
(\fL f)_{!} (x(T^{\fL f}) = (\fL f)_{!} (x (h^{*} T^{f})) =
\proj_{!} f^{\prime}_{!} (x (h^{*} T^{f})) = \proj_{!}
f^{\prime}_{!} h^* (x ( T^{f})) = \proj_{!} \ev^* f_{!} h^* (x (
T^{f})),
\end{equation}

using the naturality of the umkehr map. The first equation is true
because $x \in \Pont^* (n+1)$. We can rephrase this formula using
the notion of the transgression homomorphism.

\begin{defn}\label{deftransgress}
Let $X$ be a space, $LX$ its free loop space, $\ev: \bS^1 \times LX
\to X$ the evaluation map and $\proj: \bS^1 \times LX \to LX$ be the
projection onto the first factor. The \emph{transgression} is the
homomorphism

\[
\trg := \proj_{!} \circ \ev^*: H^* (X) \to H^{*-1} (LX).
\]
\end{defn}

Formula \ref{mmmclassesofloopspace1} becomes:

\begin{prop}\label{mmmclassesofloopspace2}
The diagram
\begin{equation}
\xymatrix{
\Pont^{k} (n+1) \ar[r] \ar[d]^{\kappa_{\fL E}} & \Pont^{k} (n) \ar[d]^{\kappa_{E}} \\
H^{k-n-1} (LX)   & H^{k-n} (X). \ar[l]_-{\trg} }
\end{equation}
is commutative.
\end{prop}

The usefulness of the above construction stems from the fact that
the transgression is injective in some cases, which we
will explain now.

\begin{prop}\label{transgressinjective}
Let $X$ be a simply-connected space such that $H^{*} (X; \bQ)\cong
\Lambda V$ is a free graded-commutative algebra on a
finite-dimensional graded vector space $V$. Then $H^*(LX; \bQ)$ is a
free-graded commutative algebra on a finite-dimensional vector space
as well and the transgression homomorphism $\tilde{H}^{*} (X; \bQ)
\to H^{*-1} (LX; \bQ)$ is injective.
\end{prop}

\begin{proof}
Of course, the transgression is not a ring homomorphism. Instead,
the following product formula holds ($\eta: LX \to X$ is the
evaluation at the basepoint):

\begin{equation}\label{transgressproduct}
\trg (x_1 x_2) = (-1)^{|x_1|} \eta^* x_1 \trg(x_2) +  \trg (x_1)
\eta^* x_2.
\end{equation}

This is shown as follows. Let $u \in H^1 (\bS^1) $ be the standard
generator. Write $\ev^* x_i = 1 \times a_i + u \times b_i \in H^*
(LX \times \bS^1$ for some $a_i,b_i \in H^* (LX)$. Then $\trg (x_i)
= \proj_{!} \ev^* x_i = b_i$ and $\eta^* x_i = a_i$. Formula
\ref{transgressproduct} follows in a straightforward manner from
Proposition \ref{gysinproperties2} (1).

Let $K(V \dual) = \prod_{k} K(V_{k}\dual;k)$ be the graded
Eilenberg-Mac Lane space. There is a tautological map $s: X \to
K(V\dual)$ which induces an isomorphism in rational cohomology
because the cohomology algebra of $X$ is free graded-commutative.
Therefore $s$ is a rational homotopy equivalence. Because $\pi_1 (X)
=0$, $LX$ is simple and there is a rational homotopy equivalence
$(LX)_{\bQ} \simeq L (X_{\bQ}) \simeq L (K(V\dual))$. The diagram

\[
\xymatrix{
H^* (X; \bQ)\ar[r]^-{\trg} & H^{*-1} (LX;\bQ)\\
H^* (K(V\dual); \bQ)\ar[r]^-{\trg} \ar[u]^{s^*} & H^{*-1} (LK(V\dual);\bQ) \ar[u]^{Ls^*}\\
}
\]

is commutative and the vertical arrows are isomorphisms. Thus we can
assume that $X = K(V\dual)$.

The map $\eta^* \oplus \trg :V \oplus \sigma^{-1} V \to H^*(LX)$
induces an algebra map $\tau: \Lambda (\eta^* V \oplus \trg (V)) \to
H^ {*} (LX)$, which is an isomorphism by the following argument.
Since a product of Eilenberg-Mac-Lane spaces is an abelian
topological group, the fibration $\Omega X \stackrel{\inc}{\to} LX
\stackrel{\eta}{\to} X$ is a product and thus it has a retraction
$r: LX \to \Omega X$. The maps $\eta^* $ and $r^*$ induce an
isomorphism $H^* (LX) \cong H^* (\Omega X) \otimes H^* (X)$.
Moreover, it is well-known that the composition $H^* (X)
\stackrel{\trg}{\to} H^{*-1} (LX) \stackrel{\inc^*}{\to} H^{*-1}
(\Omega X)$ maps $V$ to a generating subspace of the target. It
follows that $\tau$ is an epimorphism which has to be an isomorphism
by a dimension count.

A straightforward application of \ref{transgressproduct} completes
the proof: let $x_1, \ldots , x_n$ be a homogeneous basis of $V$,
$y_i := \eta^* x_i$. From \ref{transgressproduct}, one derives the
identity

\[
\trg (x_{1}^{m_1} \ldots x_{n}^{m_n}) = \sum_{i =0}^{n} m_i y_1
\ldots y_{i-1}^{m_{i-1}} y^{m_i -1}_{i} \trg(x_i) y_{i+1}^{m_{i+1}}
\ldots y_{n}^{m_n}
\]

which implies that $\trg$ is injective because the terms on the
right hand side are all linearly independent.
\end{proof}

\begin{lem}
Let $G$ be a simply connected compact Lie group. Then the spaces
$BG$, $LBG$, $L^2 BG$ satisfy the assumptions of Proposition
\ref{transgressinjective}.
\end{lem}

\begin{proof}
The case of $BG$ is a well-known result generally attributed to
Borel. Since $BG$ is $3$-connected, the spaces $LBG$ and $L^2 BG$
are simply connected and therefore the first half of the statement
of Proposition \ref{transgressinjective} can be applied.
\end{proof}

\section{Computations for $\cp^m$-bundles}\label{cpnsection}

Let $V \to X$ be an $(m+1)$-dimensional hermitian complex vector
bundle. Let $q: \bP (V) \to X$ be the projective bundle of $V$ (its
fibre is $\cp^m$ and its structural group is $\bP U(m+1)$. The
finite isogeny $SU(m+1) \to \bP U (m+1)$ induces a rational homotopy
equivalence $BSU(m+1) \to B \bP U (m+1)$. Therefore we conclude

\begin{lem}\label{ignorefirstchern}
Any characteristic class of $\cp^m$-bundles with structural group
$\bP U(m+1)$ is a polynomial in the Chern classes (recall our
indexing convention for characteristic classes) $c_4, c_6, \ldots ,
c_{2m+2}$ (i.e.: the first Chern class $c_2$ does not occur).
\end{lem}

We will use this Lemma in section \ref{holomorphic}. From now on, we
restrict our attention to hermitian vector bundles with trivialized
determinant and Lemma \ref{ignorefirstchern} tells us that we do not
loose anything. There is a tautological complex line bundle $L_V \to
\bP V$ and the first Chern class of $L_{V}\dual$ is denoted by $z_V
\in H^2 (\bP (V))$. There is a natural isomorphism

\begin{equation}\label{tangentprojective}
T_v \bP(V) \oplus \bC \cong q^* V \otimes  L_{V}\dual.
\end{equation}

Because $\cp^m = SU(m+1) / S(U(1) \times U(m))$, $S(U(1) \times
U(m)) = SU(n+1) \cap U(1) \times U(n)$, we can identify the total
space of the universal bundle $E(SU(m+1),\cp^m)$ with $B(S(U(1)
\times U(m)))$. Under this identification, the classifying map of
the vertical tangent bundle corresponds to the map induced by the
group homomorphism

\begin{equation}\label{character}
S(U(1) \times U(m)) \to U(m) \subset SO(2m); \; \begin{pmatrix} z &
0
\\ 0 & A
\end{pmatrix} \mapsto z^{-1} A.
\end{equation}

\subsection*{The Pontrjagin character for $\cp^m$-bundles}

In principle, the MMM-classes of the universal $\cp^m$-bundle
$E(SU(m+1),\cp^m)$ were computed in Hirzebruch's lecture notes
\cite{HBJ}. However, the formula that appears there is not
appropriate to show Theorems \ref{hirzebruch1} and
\ref{hirzebruch2}. Therefore we follow another path. First we turn
to the proof of Theorem \ref{hirzebruch2}, which follows immediately
from Proposition \ref{twoplustrivial} below.

Our method is to use the Leray-Hirsch Theorem for the computation of
generalized MMM-classes. Let $V \to X$ be a complex vector bundle of
rank $m+1$, $q: \bP (V) \to X$, $L_V \to \bP (V)$ and $z_V \in H^2
(\bP (V))$ as above.

As an $H^*(X)$-algebra, $H^* (\bP V)$ is isomorphic to $H^*(X)[z_V]/
(\sum_{i} c_{2i}(V) z_{V}^{m+1-i})$. The set $\{1,z_V, \ldots ,
z_{V}^{m}\}$ is a $H^* (X)$-basis of $H^*(\bP(V))$. Moreover,
$q_{!}$ is the $H^*(X)$-linear map determined by $q_{!}
(z_{V}^{i})=0$ for $0 \leq i \leq n-1$ and $q_{!} (z_{V}^{m})=1$.
The higher powers of $z_{V}$ can be expressed explicitly in terms of
this basis. This gives an algorithm to compute $q_{!}$, which is not
very manageable in general. But it is manageable if all but one of
the Chern classes of $V$ are zero.

\begin{prop}\label{twoplustrivial}
Let $X = BSU(2)$ and $V \to X$ the universal $2$-dimensional vector
bundle. Then $H^* (BSU(2)) = \bQ [u]$ where $u \in H^4 (BSU(2))$ is
the second Chern class of $V$. Consider the projective bundle $\bP
(V \oplus \bC^{m-1}) \to X$, which is a $\cp^{m}$-bundle. Then

\[
q_{!} (\ch (T_v \bP (V \oplus \bC^{m-1})))= \sum_{p=0}^{\infty} a_p u^p,
\]

where $a_p = (-1)^p \left(\frac{m-1}{(m+2p)!} + \sum_{k + l =p}
\frac{2}{(m+2k)!!(2l)!} \right) \neq 0$.
\end{prop}

\begin{proof}
By the isomorphism \ref{tangentprojective}, we obtain

\begin{equation}\label{tangentproduct}
q_{!} (\ch (T_v \bP (V \oplus \bC^{m-1}))) = (\ch (V) + m-1) q_{!}
(\ch (L_{V \oplus \bC^{m-1}}\dual)) .
\end{equation}

When we write the total Chern class of $V$ formally as $c(V) = (1 +
x_1) (1+ x_2)$, then $x:=x_1 = -x_2 $ and $u = -x^2$. Thus $\ch (V)
= \exp (x_1) + \exp (x_2) = 2 \cosh (x) = 2 \cos (\sqrt{u})$.

Let us compute $q_{!} (\ch (L_{V \oplus \bC^{m-1}}\dual))=
\sum_{l=0}^{\infty} \frac{1}{l!} q_{!} (z_{V \oplus \bC^{m-1}}^{l}
)$. With $z := z_{V \oplus \bC^{m-1}}$, we get the algebra
isomorphism

\[
H^* (\bP (V \oplus \bC^{m-1})) \cong \bQ [u,z]/(z^{m+1}+u
z^{m-1})
\]

Therefore, for $l \geq 0$,

\[
z^{m+2l+1}= (-1)^{l+1} z^{m-1} q^* u^{l+1}; \; z^{m+2l}= (-1)^{l} z^{m} q^* u^{l}.
\]

Therefore $q_{!} (z^{m+2l+1})=0$ and $q_{!} (z^{m+2l})=(-1)^l u^l$
and

\begin{equation}\label{chern}
q_{!} (\ch (L_{V \oplus \bC^{m-1}}\dual)) = \sum_{l=0}^{\infty}
\frac{(-1)^l}{(m+2l)!} u^l.
\end{equation}

Combine \ref{chern} and \ref{tangentproduct} to finish the
proof.
\end{proof}

\subsection*{The case of $\cp^2$-bundles}

Here we show Theorem \ref{hirzebruch1}. We consider the $\cp^2$-bundle $q:B S(U(1) \times U(2)) \to BSU(3)$.

\begin{prop}\label{linkernel}
Let $\cL$ be the total Hirzebruch $\cL$-class. Then $q_{!} (\cL
(T^q)) = 1 \in H^* (BSU(3))$. In particular, $\cL_{4k}$ lies in the
kernel of $\kappa_q: \Pont^{4k} (4) \to H^{4k-4} (BSU(3))$ for all
$k \geq 2$.
\end{prop}

\begin{proof}
We offer three methods since they are all interesting. The first and
most elementary method is a direct computation that can be found in
\cite{HBJ}, p.51 ff.

The second method is to use the loop space construction and then the
vanishing theorem \ref{vanishing} for the resulting $\bS^1 \times
\cp^2$-bundle.

Another method comes also from index theory. By the family index
theorem, the class $q_{!} (\cL(T^q))$ agrees with the Chern character of
the index bundle of the signature operator. Since the group $SU(3)$
acts by isometries on $\cp^2$ with respect to the Fubini-Study
metric and since $SU(3)$ is connected, this index bundle is trivial.
\end{proof}

For the proof of Theorem \ref{hirzebruch1}, we will use complex
coefficients because we are going to employ Chern-Weil theory. Let
$G$ be a compact connected Lie group with maximal torus $T$ and Weyl
group $W$. Let $f: BT \to BG$ be the universal $G/T$-bundle. Let
$\ft$ be the Lie algebra of $T$ and $\fg$ be the Lie algebra of $G$.
Recall the Chern-Weil isomorphism $CW:\Sym^* (\fg\dual_{\bC})^G
\cong H^* (BG; \bC )$, which is natural in $G$. Moreover, $\Sym^*
(\fg\dual_{\bC})^G \cong \Sym^* (\ft\dual_{\bC})^W$ by restriction.
In other words, there is a commutative diagram

\[
\xymatrix{
H^* (BG; \bC) \ar[rr]^{f^*} &  &  H^{*} (BT; \bC) \\
\Sym^* (\fg\dual_{\bC})^G \ar[r]^{\cong}_{\res} \ar[u]^{CW}_{\cong}
&\Sym^* (\ft\dual_{\bC})^W \ar[r]^{\subset} & \Sym^*
(\ft\dual_{\bC}) \ar[u]^{CW}_{\cong} }.
\]

Now we express the transfer $\trf_{f}^{*}:  H^* (BT; \bC) \to H^*
(BG; \bC)$ as a map $\Sym^* (\ft\dual_{\bC}) \to \Sym^*
(\ft\dual_{\bC})^W$.

\begin{lem}\label{transferchernweil}
As a map $\Sym^* (\ft\dual_{\bC}) \to \Sym^* (\ft\dual_{\bC})^W$,
the transfer $\trf_{f}^{*}$ agrees with the averaging operator $F
\mapsto \sum_{w \in W} w^* F$.
\end{lem}

\begin{proof}
The left $G$-action on $G/T$ commutes with the right-action of $W$.
Therefore there is a fibre-preserving right-action of $W$ on the
bundle $E(G; G/T) \to BG$. The total space $E(G;G/T)$ is homotopy
equivalent to $BT$ and the homotopy equivalence is $W$-equivariant.
In particular, $f \circ w = f$ for all $w \in W$. Therefore
$\trf_{f}^{*} = \trf_{f \circ w}^{*} = \trf_{f}^{*} \circ w^*$. In
other words, the transfer is $W$-equivariant when considered as a
map $\Sym^* (\ft\dual_{\bC}) \to \Sym^* (\ft\dual_{\bC})^W$. The
composition
\[
\Sym^* (\ft\dual_{\bC})^W \stackrel{f^*}{\to} \Sym^*
(\ft\dual_{\bC}) \stackrel{trf_{f}^{*}}{\to} \Sym^*
(\ft\dual_{\bC})^W
\]
is the map $\trf_{f}^{*} f^*$, which is $\chi (G/T) \id = |W| \id$.
The lemma follows from these two facts by elementary representation
theory of finite groups.
\end{proof}

\begin{proof}[Proof of Theorem \ref {hirzebruch1}]
Consider the diagram

\[
\xymatrix{
  & BT \ar[d]^{g}\\
BSO(4) & BS(U(1) \times U(2)) \ar[d]^{q} \ar[l]_-{h}\\
  & BSU(3), }
\]

where $h$ is the classifying map of the vertical tangent bundle and
$T$ is the standard maximal torus of $SU(3)$ (the group of diagonal
matrices of determinant $1$). Abbreviate $f:= q \circ g$. The
strategy of the proof is to show first that the kernel of

\begin{equation}\label{composition}
H^{4d+4} (BSO(4)) \stackrel{h^*}{\to} H^{4d+4} (BS(U(1) \times
U(2))) \stackrel{q_{!}}{\to} H^{4d} (BSU(3))
\end{equation}

is $2$-dimensional if $d \geq 1$. The second step will be that the
intersection of the kernel of \ref{composition} with $\Pont^{4d+4}
(4)$ has dimension $1$, generated by $\cL_{4d+4}$, which shows the
theorem.

Clearly $\dim H^{4d+4} (BSO(4)) = d+2$; we will show that the image
of the composition in \ref{composition} has dimension $d$. Let $x
\in H^{4d+4} (BSO(4))$. Write $x= C_1 (p_4,p_8) + \chi C_2
(p_4,p_8)$. Write $\cL_{4d} = a_d p_{4}^{d} + p_8 A(p_4,p_8)$ for
$a_d = 2  \frac{2^{2d} B_{2d}}{(2d)!} \neq 0\in \bQ$ and a certain
polynomial $A$. It follows that one can write $C_1 (p_4,p_8) = a
\cL_{4d} + p_8 C_3 (p_4,p_8) = a \cL_{4d} + \chi^2 C_3 (p_4,p_8)$.
In other words

\[
x = a \cL_{4d+4} + \chi ( \chi C_3 (p_4,p_8) + C_2 (p_4, p_8)) =: a
\cL_{4d+4} + \chi F(\chi,p_4)
\]

for a certain polynomial. This expression is uniquely determined.
Next we express $f_{!} (h^* (x))$ as

\[
f_{!} (h^* (x)) = f_{!} (a h^* \cL_{4d}) + f_{!} (h^* \chi h^*
F(\chi,p_4) ) = 0 + \trf_{f}^{*} (F (\chi,p_4))
\]

by Proposition \ref{linkernel} and \ref{deftransfer}. Therefore, the
image of \ref{composition} agrees with the image of the composition

\begin{equation}
H^{4d} (BSO(4)) \stackrel{h^*}{\to} H^{4d} (BS(U(1) \times U(2)))
\stackrel{\trf_{q}^{*}}{\to} H^{4d} (BSU(3))
\end{equation}

which is the same as image of

\begin{equation}\label{composition2}
H^{4d} (BSO(4)) \stackrel{(hg)^*}{\to} H^{4d} (BS(U(1) \times U(2)))
\stackrel{\trf_{f}^{*}}{\to} H^{4d} (BSU(3))
\end{equation}

because $\trf_{qg}= \trf_g \circ \trf_q$ and because $\trf_{g}^{*}
g^*$ is the multiplication with the Euler number of the fibre of
$g$, which is $2$ since $g$ is an $\bS^2$-bundle.

We write the complexified Lie algebra of $T$ as
$\ft=\{(x_1,x_2,x_3 \in \bC^3 | x_1 + x_2 + x_3 =0\}$.
The Weyl group is $\Sigma_3$, acting by the permutation representation.
We write $x_1, x_2, x_3$ for the coordinate functions on $\ft$.

Under the map $h \circ g:BT \to BSO(4)$, the elements $\chi$ and
$p_1$ are mapped by

\begin{equation}\label{mappchi}
\chi \mapsto (x_2 - x_1)(x_3 - x_1); \; p_4 \mapsto (x_2 - x_1)^2 +
(x_3 - x_1)^2;
\end{equation}

the reason is the isomorphism \ref{tangentprojective} or the
equivalent expression \ref{character}. These elements lie in the
$3$-dimensional space $V:=\Sym^2 (\ft\dual_{\bC})$ on which we now
introduce the basis

\[
z_1 = (x_2-x_1)(x_3-x_1), \; z_2 = (x_1-x_2)(x_3-x_2) , \; z_3 =
(x_2-x_3)(x_1-x_3);
\]

the Weyl group acts by permutations on that basis. Rewriting
\ref{mappchi} yields

\begin{equation}\label{mappchi2}
\chi \mapsto z_1 ; p_4 \mapsto 2z_1 + z_2 + z_3 = :z_1 + s_1;
\end{equation}

where $s_i$ denotes the $W$-invariant element $s_i:= z_{1}^{i} +
z_{2}^{i} + z_{3}^{i} $.

In view of \ref{transferchernweil}, we have to show that the image
of the $(d+1)$-dimensional subspace $X:= \spann \{ z_{1}^{k} (z_1 +
s_1)^{d-k} \}_{k=0, \ldots , d}$ of $\Sym^k V$ under the averaging
operator $\Phi= \sum_{\sigma \in \Sigma_3} \sigma$ has dimension
$d$. To this end, abbreviate $v_{k,d}= z_{1}^{k} (z_1 + s_1)^{d-k}$
and note that

\[
\Phi (v_{k,d}) = 2 \sum_{j} \binom{d-k}{j} s_{1}^{j} s_{d-j}.
\]

Let $C$ be the $(d+1) \times (d+1)$-matrix with entries $c_{j,k}=
\binom{d-k}{j}$ ($0 \leq j,k \leq d$); $C$ is nonsingular because
its entries below the antidiagonal are zero and the entries on the
antidiagonal are $1$, therefore $\det (C) = \pm 1$. Therefore the
equation

\[
\Phi (\sum_k a_k v_{k,d}) = s_{1}^{j} s_{d-j}
\]

has a solution $(a_k)$ in $\bC^{d+1}$. Therefore, the image of $X$
under $\Phi$ contains the elements

\[s_{1}^{d}, s_{1}^{d-1} s_1, s_{1}^{d-2} s_2, \ldots s_1 s_{d-1},
s_d.
\]

We claim that these polynomials span an $d$-dimensional vector space
and show this claim by induction on $d$. The case $d=2$ is trivial.

Because the multiplication by $s_1$ is injective, it suffices to
show that $s_d$ is not a linear combination of $s_{1}^{d},
s_{1}^{d-1} s_1, s_{1}^{d-2} s_2, \ldots s_1 s_{d-1}$. Assume, to
the contrary that

\[
s_d (z_1,z_2,z_3) = \sum_{j=0}^{d-1} c_j s_j(z_1,z_2,z_3)
s_{1}^{d-j} (z_1,z_2,z_3), \; c_j \in \bC.
\]

Restricting to the subspace defined by $z_1 + z_2 + z_3=0$, we get
the equation

\[
z_{1}^{d} + z_{2}^{d} + (-z_{3} - z_{2})^{d} = \sum_{j=0}^{d-1} c_j
s_j(z_1,z_2,z_3) s(z_1+z_2+z_3)^{d-j}=0
\]

which is obviously wrong for all $d \geq 2$. This finishes the proof
that \ref{composition} has a $2$-dimensional kernel.

One element in this kernel is $\cL_{4d+4}$. Another element in $(p_4
- \chi)^{d+1}$. To see this, look at \ref{mappchi2}: $g^* h^*(p_4 -
\chi)=s_1 \in V^{\Sigma_3}=\im f^*$. Since $g^*$ is injective, it
follows that $h^* (p_4 - \chi) = q^* y$ for a certain $y$. It
follows that

\[
q_{!} (h^* (p_4 - \chi)^{d+1}) = q_0{!} (q^* y^{d+1} 1) = y^{d+1}
q_{!}(1) =0.
\]

This means that any element in the kernel of \ref{composition} can
be written as $a_1 \cL_{4d+4} + a_2 (p_4 - \chi)^{d+1}$. This
belongs to $\Pont^{4d+4} (4)$ if and only if $a_2=0$.
\end{proof}

\section{From linear to algebraic
independence}\label{lineartoalgebraicsection}

In this section, we show Theorem \ref{maintheorem2}, based on
Theorem \ref{maintheorem1} whose proof we just completed. It is this
step where we have to sacrifice the connectedness of the manifolds.
The main step is:

\begin{prop}\label{lineartoalgebraic}
Let $W \subset \sigma^{-n} H^{*}(BSO(n); \bQ)$ be a linear subspace
such that $\kappa^n :W \to H^{*} (\coprod_{M \in \cR} B \Diff^+
(M)_{+}; \bQ)$ is injective. Then the extension $\Lambda
\kappa^{n,0}:\Lambda W \to H^{*} (\coprod_{M \in \cR} B \Diff^+
(M)_{+}; \bQ)$ is injective.
\end{prop}

Assuming Proposition \ref{lineartoalgebraic} for the moment, we can
show Theorem \ref{maintheorem2}.

\begin{proof}[Proof of Theorem \ref{maintheorem2}:]
If $n$ is even, then Theorem \ref{maintheorem2} is an immediate
consequence of Theorem \ref{maintheorem1} and Proposition
\ref{lineartoalgebraic}.

If $n$ is odd, we need a little argument. If $W \subset V$ are
graded vector spaces, then $\Lambda (V / W) \cong \Lambda (V) /(W)$,
where $(W)$ is the $2$-sided ideal generated by $W$. Let $V:=
\sigma^{-n} H^* (BSO(n))$ and let $W$ be the span of the Hirzebruch
$\cL$-classes. Choose a complement $U \subset V$ of $W$. By Theorem
\ref{maintheorem1}, $\kappa^{n,0} : U \to H^* (\coprod_{M \in \cR_n}
B \Diff^+ (M))$ is injective; whence $\Lambda (U) \to H^*
(\coprod_{M \in \cR_n} B \Diff^+ (M))$ is injective by Proposition
\ref{lineartoalgebraic}. But $\Lambda( U) \cong \Lambda (V)/ (W)$
and therefore the statement follows.
\end{proof}

\begin{proof}[Proof of Proposition \ref{lineartoalgebraic}]
Without loss of generality, we can assume that $W$ is
finite-dimensional.

There exist connected $n$-manifolds $M_1, \ldots M_r$ such that
$\kappa^{n}: W \to H^{*} (\cB_{n}^{0}) \to H^{*} (\coprod_{i=1}^{r}
B \Diff^+ (M_i))$ is injective. Put $M:= \coprod_{i=1}^{r } M_i$.
The group $\prod_{i=1}^{r} \Diff^+ (M_i)$ acts on $M$ separately on
each factor. Thus there is a smooth $M$-bundle $E \to B=
\prod_{i=1}^{r} B \Diff^+ (M_i)$. The diagram

\[
\xymatrix{ W   \ar[d]^{\kappa_E} \ar[r]^{\kappa^{n,0}} & H^* (\cB_{n}^{0}) \ar[d]\\
H^* (\prod_{i=1}^{r} B \Diff^+ (M_i)) \ar[r] &  H^*
(\coprod_{i=1}^{r} B \Diff^+ (M_i))}
\]

(the bottom map comes from the natural map) commutes and therefore
$\kappa_E: W \to H^* (B)$ is injective. The purpose of this argument
is to show that we can find a single manifold $M$ and a smooth
$M$-bundle $f:E \to B$ on a connected base space such that
$\kappa_E: W \to H^* (B)$ is injective.

Let $m \in \bN$ and let $\Sigma_m$ be the symmetric group. Now we
consider

\begin{equation}
\xymatrix{
E \dash \ar[r]^{p \dash} \ar[d]^{f \dash} & E \ar[d]^{f}\\
E(\Sigma_m; \underline{m} \times B^m) \ar[r]^-{p} \ar[d]^{q } & B \\
E(\Sigma_m; B^m); & \\
}
\end{equation}

the map $p$ is given by the $\Sigma_m$-equivariant map
$\underline{m} \times \ni (i,x_1,\ldots , x_m) \mapsto x_i \in B$;
the square is a pullback and the composition $q \circ f \dash$ is a
smooth $\underline{m} \times M$-bundle (note the similarity to the
loop space construction).

In the same way as in \ref{mmmclassesofloopspace2}, one sees that
the diagram

\begin{equation}\label{diagram1}
\xymatrix{ W \ar[r]^{\kappa_E} \ar[dr]^{\kappa_{E \dash}} & H^* (B)
\ar[d]^{q_{!} \circ p^*}\\
 & H^* (E( \Sigma_m; B^m)) }
\end{equation}

commutes. Hence the induced diagram

\begin{equation}\label{diagram2}
\xymatrix{\Lambda W \ar[r]^{\Lambda \kappa_E} \ar[dr]^{\Lambda
\kappa_{E \dash}} & \Lambda H^* (B)
\ar[d]^{\Lambda(q_{!} \circ p^*)}\\
 & H^* (E( \Sigma_m; B^m)) }
\end{equation}

commutes as well. The top horizontal map is injective by assumption.

The map $q_{!} \circ p^* : H^* (B; \bQ) \to H^* (E( \Sigma_m; B^m);
\bQ)$ induces an algebra map $\Lambda \tilde{H}^* (B) \to H^* (E(
\Sigma_m; B^m); \bQ)$ which is an isomorphism up to degree $m/2$.
This is a combination of the Barratt-Priddy-Quillen theorem and
homological stability for symmetric groups (Nakaoka et alii). See
\cite{EG} for details and references.
\end{proof}

It is obvious that it is necessary to consider nonconnected manifolds in the above proof of Theorem \ref{mainthm2}.
We do not know whether Theorem \ref{maintheorem2} remains true if
$\Lambda \kappa^n$ is replaced by $\Lambda \kappa^{n,0}$. In the
$2$-dimensional case, the situation is different. All published
proofs of Theorem \ref{morita} show that $ \Lambda \kappa^{n,0}$ is
injective. For the passage from $\kappa^{2,0}$ to $\Lambda
\kappa^{2,0}$, the use of Harer's homological stability theorem for
the mapping class groups is essential, while the stability result is
not necessary to show that $\kappa^{2,0}$ is injective (this point
is most obvious in Miller's proof \cite{Mil}). Since a large portion
of the proof of Theorem \ref{maintheorem1} relies on the
$2$-dimensional case, there are partial results for $\Lambda
\kappa^{n,0}$, see e.g. \cite{Gian} for a result in the
$4$-dimensional case.

\section{The holomorphic case}\label{holomorphic}

In this section, we prove Theorem \ref{holomorphiccase}, which is
parallel to the proofs of \ref{maintheorem1} and \ref{maintheorem2}.
So we sketch only the differences.

The proofs of \ref{morita} given by Miller and Morita show that
Theorem \ref{holomorphiccase} holds if $m=1$. The inductive
procedure works in the same way; the proof of \ref{inductionstep} is
easily adjusted and shows that only the classes of the form
$\kappa_E (\ch_{2d})$, $2d \geq 2m$ cannot be detected on products.

If $q: E \to BU(m+1)$ is the universal $\cp^m  $-bundle, then the
class $q_{!} (\ch_{2d} (T_v E))$ is nonzero if $2d \geq 2m$ and $d-m
\equiv 0 \pmod 2$ by Theorem \ref{hirzebruch1}. Of course, $BU(m)$
is not a complex manifolds; nevertheless it can be approximated by
the Grassmann manifolds $\Gr_m (\bC^r)$ of $m$-dimensional quotients
of $\bC^r$ for $r \gg m$, which is a projective variety. The
tautological vector bundle on $\Gr_m (\bC^r)$ is a holomorphic
vector bundle and hence its projectivization is a holomorphic fibre
bundle.

Thus we are left with showing that $\kappa^{m}_{\bC}(\ch_{2d}) \neq 0$ if $2m \leq 2d$ and $d-m \equiv 1 \pmod 2$. The loop space
construction as in section \ref{loopspacesection} does not make
sense in the holomorphic realm. One could replace $\bS^1$ by $\cp^1$
in the loop space construction and the space $\map (\bS^1,BU(m+1))$
by the approximating space $\hol_k (\cp^1; \Gr_{m+1}(\bC^r)$ and
then use the fact (proven by Segal and Kirwan) that the space of
holomorphic maps into a Grassmannian is a good homotopical
approximation to the space of all maps, bu we prefer a more direct
route. Let $T \to \cp^1$ and $L \to
\cp^r$ be the tautological line bundles. Consider the $2$-dimensional vector bundle $V=(\bC \oplus T)
\boxtimes L \to \cp^1 \times \cp^r$. Its total Chern class is $c(V)=
(1 + 1 \times x)(1+ z \times 1+ 1 \times x)$, where $x \in H^2
(\cp^r )$ and $z \in H^2(\cp^1)$ are the usual generators. Therefore
the second Chern class is $ u=1 \times x^2 + z \times x$ and $u^l =1
\times x ^{2l}+ l z \times x^{2l-1} \neq 0$ for $r \gg 2l$. Consider
the composite bundle

\[
\bP (V \oplus \bC^{m-2}) \stackrel{q}{\to} \cp^1 \times \cp^r
\stackrel{\proj}{\to} \cp^r
\]

with fibre $\cp^1 \times \cp^{m-1}$. A computation similar to the one in \ref{mmmclassesofloopspace1} (and using Proposition \ref{Gysinproperties2}, (4)) shows that

\[
\proj_{!} q_{!} (\ch (T^{proj \circ q})) = \proj_{!} q_{!} (\ch (T^q)) + \proj_{!} (q_{!} q^* \ch (T^{proj})) = \proj_{!} q_{!} (\ch (T^q)).
\]

By Theorem \ref{twoplustrivial} and Lemma
\ref{ignorefirstchern}

\[
\proj_{!} q_{!} (\ch (T^q)) = \proj_{!} (\sum_{l=0}^{\infty} a_l \proj_{!} u^l),
\]

where $a_l$ is the nonzero rational number from Theorem \ref{twoplustrivial}. But $\proj_{!} (u^l)= l x^{2l-1}$. This
finishes the proof that $\kappa_{E}{\bC}(\ch_{2d}) \neq 0$ for a
certain $m$-dimensional bundle with $m - d$ odd.

To show the second half of Theorem \ref{holomorphic}, we replace the
space $E \Sigma_m$ by the configuration space $C^m ( \bC^r)$ of $m$
numbered points in $\bC^r$ for sufficiently large $r$.

\appendix

\section{Gysin maps and the transfer}\label{gysintransfersection}

Here we give a brief recapitulation of Gysin maps for fibre bundles.
The Gysin homomorphism of a smooth closed oriented manifold bundle
is defined by means of the Leray-Serre spectral sequence, see e.g.
\cite{MorGCC}, p. 147 ff. Let $E^{p,q}_{r}$ be the Leray-Serre
spectral sequence. The Gysin map $f_{!}$ is defined as the
composition

\begin{equation}\label{gysinlerayserre}
f_{!}:H^{k+n} (E) \to E_{\infty}^{k,n} \subset E_{2}^{k,n} = H^{k}
(B; \underline{H^n (M)}) \stackrel{\cap [M]}{\to} H^{k} (B).
\end{equation}

The last map arises as follows. Because the bundle is oriented, the
fundamental class $[M]$ of the fibre defines a homomorphism
$\underline{H^n (M)} \to \bZ$ of coefficient systems on $B$ (the
system $\bZ$ is the constant one); it is always an epimorphism and
it is an isomorphism if $M$ is connected. One can replace $\bZ$ of
course by any other ground ring. Below there is a list of the main properties of the Gysin map. The proof can be found in \cite{BH}, section 8.

\begin{prop}\label{gysinproperties1}
Let $M$ be a closed oriented $n$-manifold and $f:E \to B$ be smooth
oriented $M$-bundle.
\begin{enumerate}
\item \emph{Naturality:} If
\[
\xymatrix{ E \dash \ar[r]^{\hat{g} } \ar[d]^{f \dash} & E
\ar[d]^{f}\\
B \dash \ar[r]^{g} & B }
\]
is a pullback-square, then $f_{!} \dash \circ \hat{g}^{*} = g^*
\circ f_{!}$.
\item \emph{$H^*(B)$-linearity:} If $x \in H^* (E)$ and $y \in H^* (B)$, then $f_{!} ( (f^*y ) x)
= y  \times f_{!} (x)$.
\item \emph{Normalization:} If $M$ is an oriented $n$-manifold with fundamental class
$[M] \in H_n (M)$ and $f: M \to *$ the constant map, then $f_{!} (x) = \langle x;[M] \rangle 1$ for all $ x \in H^* (M)$.
\item \emph{Transitivity:} If $N$ is another closed oriented manifold
and $g: X \to E$ be a smooth oriented $N$-bundle, then $(f \circ
g)_! = f_! \circ g_!$.
\end{enumerate}
\end{prop}

The following properties are straightforward consequences of Proposition \ref{gysinproperties1}.

\begin{prop}\label{gysinproperties2}
Let $f: E \to B$ be an oriented smooth $n$-manifold bundle.
\begin{enumerate}
\item  Then $f_{!} (x f^* (y))= (-1)^{|x||y|} f_{!} ( f^*
(y)x) =   (-1)^{(|x|-n)|y| }   f_{!} (x) y$.
\item Let $f_i:E_i \to B_i$, $i=1,2$, be two oriented fibre bundles of
fibre dimension $n_i$. Consider the oriented fibre bundle $f = f_1
\times f_2: E = E_1 \times E_2 \to B = B_1 \times B_2$ of fibre
dimension $n= n_1 + n_2$. Then $(f_1 \times f_2)_{!} (x_1 \times
x_2) = (-1)^{n_2 |x|}(f_1)_{!} (x_1) \times (f_2)_{!} (x_2)$ for all
$x_i \in H^* (E_i)$.
\item If $f: E \to B$ is a homeomorphism (the fibre is a point),
then $f_{!} = (f^{-1})^*$.
\item If the fibres of $f$ have positive dimension, then $f_{!} \circ f^* =0$.
\end{enumerate}
\end{prop}

Another construction of Gysin maps is homotopy-theoretic in nature
and uses the Pontrjagin-Thom construction, see e.g. \cite{BG}. If
$f:E \to B$ is a smooth manifold bundle with closed fibres, then the
Pontrjagin-Thom map is a map $\PT_f: \suspinf B_+ \to \bTh (-T_v E)$
of spectra ($\bTh (-T_v E)$ is the Thom spectrum of the stable
vector bundle $-T_v E$. The Thom isomorphism of $T_v E$ is an
isomorphism $\thom: H^{* } (\suspinf E_+) \cong H^{*-n} (\bTh (-T_v
E))$ and the Gysin map is the composition

\[
f_{!} = \PT_{f}^{*} \circ \thom:H^{k} (E) = H^k (\suspinf E_+) \stackrel{\thom}{\to} H^{k-n} (\bTh (-T_v E)) \stackrel{\PT_{f}^{*}}{\to}
H^{k-n} (\suspinf B_+ ) = H^{n-k} (B).
\]

Closely related to the Gysin map is the \emph{transfer}.

\begin{defn}\label{deftransfer}
Let $f: E \to B$ be an oriented smooth bundle. Then the
\emph{transfer} is the map $\trf_{f}^{*}: H^* (E) \to H^* (B)$ given
by $\trf_{f}^{*} (x) := f_{!} (\chi (T_v E) x)$.
\end{defn}

Note that both $\chi (T_v E)$ and $f_{!}$ reverse their sign if the
orientation of $T_v E$ is reversed, so $\trf_{f}^{*}$ does not
depend on the orientation. In fact, $\trf_{f}^{*}$ is induced by a
stable homotopy class $\trf_{f}: \suspinf B_+ \to \suspinf E_+$
which only depends on the bundle and not on the orientation; in fact this stable homotopy class can be defined for more general bundles than we consider here. We do
not use this homotopy-theoretic perspective in this
paper\footnote{We use it implicitly in the proof of Theorem
\ref{maintheorem2}, though.}. What we need to know are the following properties which
are straightforward consequences of Proposition \ref{gysinproperties1}.

\begin{prop}\label{transferproperties}
Let $f: E \to B$ be an oriented smooth $n$-manifold bundle.
\begin{enumerate}
\item If $g: F \to E$ is another smooth oriented manifold bundle,
then $\trf^{*}_{f \circ g} = \trf^{*}_{f} \circ \trf^{*}_{g}$.
\item The composition $\trf_{f}^{*} f^* : H^* (B) \to H^* (B)$ is
multiplication by the Euler number $\chi (M)$ of the fibre.
\item If $f: E \to B$ is a homeomorphism (the fibre is a point),
then $\trf_{f}^{*} = (f^{-1})^*$.
\end{enumerate}
\end{prop}

\address
\email
\end{document}